\documentclass[12pt]{amsart}

\usepackage{hyperref} 
\usepackage[utf8]{inputenc}
\usepackage[T1]{fontenc}
\usepackage[english,french]{babel}
\usepackage{mathdots}
\usepackage{tikz}
\usepackage{amsmath}

\newtheorem{theorem}{Theorem}
\newtheorem{conjecture}[theorem]{Conjecture}

\newtheorem{lemma}[theorem]{Lemma}

{\theoremstyle{definition} \newtheorem{definition}[theorem]{Definition}}
{\theoremstyle{definition} }
{\theoremstyle{definition} }
{\theoremstyle{definition} }
{\theoremstyle{remark} \newtheorem{remark}[theorem]{Remark}}

\title[A generalization of the Kreweras triangle]{A generalization of the Kreweras triangle through the universal $\text{sl}_2$ weight system}
\author{Ange Bigeni}
\thanks{National Research University Higher School of Economics, Faculty of Mathematics, Usacheva
str. 6, 119048, Moscow, Russia. \href{mailto:abigeni@hse.ru}{abigeni@hse.ru}}

\begin{document}

\selectlanguage{english}

\begin{abstract}
In the theory of finite order knot invariants, the universal $\text{sl}_2$ weight system maps the chord diagrams to polynomials in a single variable with integer coefficients. In this paper, we define a family of polynomials that generalize the Kreweras triangle (known to refine the normalized median Genocchi numbers), and we show how it appears in this weight system.
\end{abstract}

\maketitle

\section*{Notations} For all pair of integers $n < m$, the set $\{n,n+1,\hdots,m\}$ is denoted by $[n,m]$, and the set $[1,n]$ by $[n]$. The set of the permutations of $[n]$ is denoted by $\mathfrak{S}_n$.

\section{Introduction}

\subsection{About the chord diagrams and the universal $\text{sl}_2$ weight system}

Let $n$ be a positive integer. In the theory of finite order knot invariants (see \cite{CDM,LZ}), a chord diagram of order $n$, or $n$-chord diagram, is an oriented circle with $2n$ distinct points paired into $n$ disjoint pairs named chords, considered up to orientation-preserving diffeomorphisms of the circle. It can be assimilated into a tuple $((p_i,p_i^*) : i \in [n])$ such that $\{p_1,p_1^*,p_2,p_2^*,\hdots\}~=~[2n]$ with $p_1 < p_2 < \hdots < p_n$ and $p_i < p_i^*$ for all $i$ : for such a tuple, the corresponding chord diagram is obtained by labelling $2n$ points on a circle with the consecutive labels $1,2,\hdots,2n$ (following the counterclockwise direction), and pairing the points labelled with $p_i$ and $p_i^*$ for all $i \in [n]$. For example, the tuples $((1,3),(2,5),(4,6))$ and $((1,5),(2,4),(3,6))$ are two representations of the $3$-chord diagram depicted in Figure~\ref{fig:examplechorddiagram}.
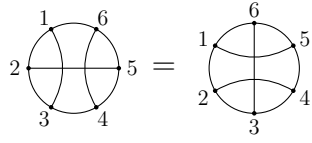
\begin{figure}[h]
\begin{center}

\begin{tikzpicture}[scale=0.6]
 
\begin{scope}[rotate=0]

\draw (0:1) arc (0:360:1);

\fill (0:1) circle(0.05);
\fill (60:1) circle(0.05);
\fill (120:1) circle(0.05);
\fill (180:1) circle(0.05);
\fill (240:1) circle(0.05);
\fill (300:1) circle(0.05);

\draw (0:1) to (180:1);
\draw (60:1) to[bend right] (300:1);
\draw (120:1) to[bend left] (240:1);

\draw (0:1.3) node[scale=0.7] {$5$};
\draw (60:1.3) node[scale=0.7] {$6$};
\draw (120:1.3) node[scale=0.7] {$1$};
\draw (180:1.3) node[scale=0.7] {$2$};
\draw (240:1.3) node[scale=0.7] {$3$};
\draw (300:1.3) node[scale=0.7] {$4$};

\end{scope} 

\begin{scope}[xshift=2cm]
\draw (0,0) node[scale=1] {$=$};
\end{scope}

\begin{scope}[xshift = 4cm, rotate=90]
\draw (0:1) arc (0:360:1);

\fill (0:1) circle(0.05);
\fill (60:1) circle(0.05);
\fill (120:1) circle(0.05);
\fill (180:1) circle(0.05);
\fill (240:1) circle(0.05);
\fill (300:1) circle(0.05);

\draw (0:1) to (180:1);
\draw (60:1) to[bend right] (300:1);
\draw (120:1) to[bend left] (240:1);

\draw (0:1.3) node[scale=0.7] {$6$};
\draw (60:1.3) node[scale=0.7] {$1$};
\draw (120:1.3) node[scale=0.7] {$2$};
\draw (180:1.3) node[scale=0.7] {$3$};
\draw (240:1.3) node[scale=0.7] {$4$};
\draw (300:1.3) node[scale=0.7] {$5$};

\end{scope} 

\end{tikzpicture}
\end{center}
\caption{Two distinct labellings of a $3$-chord diagram.}
\label{fig:examplechorddiagram}
\end{figure}

A \textit{weight system} is a function $f$ on the chord diagrams that satisfies the 4-term relations depicted in Figure \ref{fig:4term}. 

\begin{figure}[h]

\begin{center}

\begin{tikzpicture}[scale=0.6]

\begin{scope}
\draw (-1.4,0) node[scale=1.5] {$f($};
\draw (1.15,0) node[scale=1.5] {$)$};
 
\begin{scope}[rotate=120]

\draw (0:1) arc (0:60:1) ;
\draw (60:1) [dashed] arc (60:120:1) ;
\draw (120:1) arc (120:180:1) ;
\draw (180:1) [dashed] arc (180:240:1) ;
\draw (240:1) arc (240:300:1) ;
\draw (300:1) [dashed] arc (300:360:1) ;

\fill (30:1) circle(0.05);
\fill (140:1) circle(0.05);
\fill (160:1) circle(0.05);
\fill (270:1) circle(0.05);

\draw (30:1) to[bend left] (160:1);
\draw (140:1) to[bend left] (270:1);

 \fill[fill=gray, opacity=0.3]
    (0,0) -- (60:1) arc (60:120:1cm);
 \fill[fill=gray, opacity=0.3]
    (0,0) -- (180:1) arc (180:240:1cm);
     \fill[fill=gray, opacity=0.3]
    (0,0) -- (300:1) arc (300:360:1cm);
\end{scope} 
\end{scope}

\begin{scope}[xshift=3.9cm]
\draw (-1.8,0) node[scale=1.5] {$-f($};
\draw (2.3,0) node[scale=1.5] {$)=f($};
 
\begin{scope}[rotate=120]

\draw (0:1) arc (0:60:1) ;
\draw (60:1) [dashed] arc (60:120:1) ;
\draw (120:1) arc (120:180:1) ;
\draw (180:1) [dashed] arc (180:240:1) ;
\draw (240:1) arc (240:300:1) ;
\draw (300:1) [dashed] arc (300:360:1) ;

\fill (30:1) circle(0.05);
\fill (140:1) circle(0.05);
\fill (160:1) circle(0.05);
\fill (270:1) circle(0.05);

\draw (30:1) to[bend left] (140:1);
\draw (160:1) to[bend left] (270:1);

 \fill[fill=gray, opacity=0.3]
    (0,0) -- (60:1) arc (60:120:1cm);
 \fill[fill=gray, opacity=0.3]
    (0,0) -- (180:1) arc (180:240:1cm);
     \fill[fill=gray, opacity=0.3]
    (0,0) -- (300:1) arc (300:360:1cm);
\end{scope} 
\end{scope}

\begin{scope}[xshift=8.5cm]
\draw (2.25,0) node[scale=1.5] {$)-f($};
 
\begin{scope}[rotate=240]

\draw (0:1) arc (0:60:1) ;
\draw (60:1) [dashed] arc (60:120:1) ;
\draw (120:1) arc (120:180:1) ;
\draw (180:1) [dashed] arc (180:240:1) ;
\draw (240:1) arc (240:300:1) ;
\draw (300:1) [dashed] arc (300:360:1) ;

\fill (30:1) circle(0.05);
\fill (140:1) circle(0.05);
\fill (160:1) circle(0.05);
\fill (270:1) circle(0.05);

\draw (30:1) to[bend left] (160:1);
\draw (140:1) to[bend left] (270:1);

 \fill[fill=gray, opacity=0.3]
    (0,0) -- (60:1) arc (60:120:1cm);
 \fill[fill=gray, opacity=0.3]
    (0,0) -- (180:1) arc (180:240:1cm);
     \fill[fill=gray, opacity=0.3]
    (0,0) -- (300:1) arc (300:360:1cm);
\end{scope} 
\end{scope}

\begin{scope}[xshift=13cm]
\draw (1.1,0) node[scale=1.5] {$)$};
 
\begin{scope}[rotate=240]

\draw (0:1) arc (0:60:1) ;
\draw (60:1) [dashed] arc (60:120:1) ;
\draw (120:1) arc (120:180:1) ;
\draw (180:1) [dashed] arc (180:240:1) ;
\draw (240:1) arc (240:300:1) ;
\draw (300:1) [dashed] arc (300:360:1) ;

\fill (30:1) circle(0.05);
\fill (140:1) circle(0.05);
\fill (160:1) circle(0.05);
\fill (270:1) circle(0.05);

\draw (30:1) to[bend left] (140:1);
\draw (160:1) to[bend left] (270:1);

 \fill[fill=gray, opacity=0.3]
    (0,0) -- (60:1) arc (60:120:1cm);
 \fill[fill=gray, opacity=0.3]
    (0,0) -- (180:1) arc (180:240:1cm);
     \fill[fill=gray, opacity=0.3]
    (0,0) -- (300:1) arc (300:360:1cm);
\end{scope} 
\end{scope}

\end{tikzpicture}

\end{center}

\caption{The 4-term relations.}
\label{fig:4term}

\end{figure}

The theory provides the construction of nontrivial weight systems from semisimple Lie algebras, among which the Lie algebra $\text{sl}_2$ of the
$2 \times 2$ matrices whose trace is zero, which raises a weight system $\varphi_{\text{sl}_2}$ mapping the $n$-chord diagrams to elements of $\mathbb{Z}[x]$ with degree $n$. With precisions, it gives birth to a family of weight systems $(\varphi_{\text{sl}_2,\lambda})_{\lambda \in \mathbb{R}}$ related by the following equation for all $\lambda \in \mathbb{R}$ and for all $n$-chord diagram $\mathcal{D}$~:
\begin{equation}
\label{eq:relationsbetweenvarphi}
\lambda^n \varphi_{\text{sl}_2,\lambda}(\mathcal{D})(x/ \lambda) = \varphi_{\text{sl}_2}(\mathcal{D})(x).
\end{equation}
In the rest of this paper, we consider the weight system
$$\varphi = \varphi_{\text{sl}_2,2}.$$
The following is a combinatorial definition of $\varphi$ from Chmutov and Varchenko~\cite{CV}.

\begin{definition}
\label{defi:varphi}
Let $\mathcal{D}$ be an $n$-chord diagram. The weight $\varphi(\mathcal{D})$ is defined as $x$ if $\mathcal{D}$ is the unique $1$-chord diagram $\mathcal{D}_1 = ((1,2))$, otherwise $n \geq 2$ and $\varphi(\mathcal{D})$ is defined by the following induction formula~:
\begin{equation}
\label{eq:definitionvarphi}
\varphi(\mathcal{D}) = (x-k)\varphi(\mathcal{D}_a) + \sum_{\{i,j\} \subset I_a} \Delta_{i,j}(\mathcal{D}_a)
\end{equation}
where, if $\mathcal{D} = ((p_i,p_i^*) : i \in [n])$~:
\begin{itemize}
\item $a$ is any given chord $(p_i,p_i^*)$ of $\mathcal{D}$ (it is then a nontrivial result that this definition does not depend on the choice of $a$);
\item $\mathcal{D}_a$ is the $(n-1)$-chord diagram obtained from $\mathcal{D}$ by deleting the chord $a$;
\item $k$ is the cardinality $\#I_a$ where $I_a$ is the set of the integers $i \in [n]$ such that the point $p_i$ is located in the left half-plane defined by the support of $a$, and such that the chord $(p_i,p_i^*)$ intersects $a$;
\item for all $\{i,j\} \subset I_a$, $\Delta_{i,j}(\mathcal{D}_a) = \varphi(\mathcal{D}_{i,j}^1) - \varphi(\mathcal{D}_{i,j}^2)$ where $\mathcal{D}_{i,j}^1$ (respectively $\mathcal{D}_{i,j}^2$) is the $(n-1)$-chord diagram obtained from $\mathcal{D}_a$ by replacing the chords $(p_i,p_i^*)$ and $(p_j,p_j^*)$ by $(p_i,p_j)$ and $(p_i^*,p_j^*)$ (respectively by $(p_i,p_j^*)$ and $(p_j,p_i^*)$).
\end{itemize}
\end{definition}

\begin{remark}
\label{rem:divisibleparx}
For all $n$-chord diagram $\mathcal{D}$, it is straightforward, by induction on $n$, that $\varphi(\mathcal{D})$ is a monic polynomial with degree $n$ and integer coefficients, divisible by~$x$.
\end{remark}

For example, there are two $2$-chord diagrams :
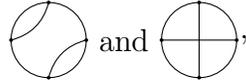
\begin{center}
\begin{tikzpicture}[scale=0.5]
 
\begin{scope}[rotate=90]

\draw (0:1) arc (0:360:1);

\fill (0:1) circle(0.05);
\fill (90:1) circle(0.05);
\fill (180:1) circle(0.05);
\fill (270:1) circle(0.05);

\draw (0:1) to[bend left] (90:1);
\draw (180:1) to[bend left] (270:1);

\end{scope} 

\draw (2,0) node[scale=1] {and};

\begin{scope}[xshift = 4cm, rotate=90]

\draw (0:1) arc (0:360:1);

\fill (0:1) circle(0.05);
\fill (90:1) circle(0.05);
\fill (180:1) circle(0.05);
\fill (270:1) circle(0.05);

\draw (0:1) to (180:1);
\draw (90:1) to (270:1);

\end{scope} 
\draw (5.2,0) node[scale=1] {,};
\end{tikzpicture}

\end{center}
and their respective weights are $x^2$ and $(x-1)x$. To compute the weight of the $3$-chord diagram $\mathcal{D}$ depicted hereafter
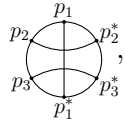
\begin{center}
\begin{tikzpicture}[scale=0.5]

\begin{scope}[xshift = 4cm, rotate=90]

\draw (0:1) arc (0:360:1);

\fill (0:1) circle(0.05);
\fill (60:1) circle(0.05);
\fill (120:1) circle(0.05);
\fill (180:1) circle(0.05);
\fill (240:1) circle(0.05);
\fill (300:1) circle(0.05);

\draw (0:1) to (180:1);
\draw (60:1) to[bend right] (300:1);
\draw (120:1) to[bend left] (240:1);

\draw (0:1.3) node[scale=0.7] {$p_1$};
\draw (60:1.35) node[scale=0.7] {$p_2$};
\draw (120:1.3) node[scale=0.7] {$p_3$};
\draw (180:1.3) node[scale=0.7] {$p_1^*$};
\draw (240:1.4) node[scale=0.7] {$p_3^*$};
\draw (300:1.4) node[scale=0.7] {$p_2^*$};

\end{scope} 
\draw (5.5,0) node[scale=1] {,};
\end{tikzpicture}

\end{center}
one can consider the chord $a = (p_1,p_1^*)$ to obtain
\begin{center}
\begin{tikzpicture}[scale=0.5]

\draw (0:0) node[scale=1] {$\varphi(\mathcal{D}) = (x-2) \varphi($};

\begin{scope}[xshift = 4.3cm, rotate=90]

\draw (0:1) arc (0:360:1);

\fill (60:1) circle(0.05);
\fill (120:1) circle(0.05);

\fill (240:1) circle(0.05);
\fill (300:1) circle(0.05);

\draw (60:1) to[bend right] (300:1);
\draw (120:1) to[bend left] (240:1);

\draw (60:1.3) node[scale=0.7] {$p_2$};
\draw (120:1.3) node[scale=0.7] {$p_3$};

\draw (240:1.4) node[scale=0.7] {$p_3^*$};
\draw (300:1.4) node[scale=0.7] {$p_2^*$};

\end{scope} 
\draw (6.9,0) node[scale=1] {$) + \varphi($};

\begin{scope}[xshift = 9.4cm, rotate=90]

\draw (0:1) arc (0:360:1);

\fill (60:1) circle(0.05);
\fill (120:1) circle(0.05);

\fill (240:1) circle(0.05);
\fill (300:1) circle(0.05);

\draw (60:1) to[bend left] (120:1);
\draw (240:1) to[bend left] (300:1);

\draw (60:1.3) node[scale=0.7] {$p_2$};
\draw (120:1.3) node[scale=0.7] {$p_3$};

\draw (240:1.4) node[scale=0.7] {$p_3^*$};
\draw (300:1.4) node[scale=0.7] {$p_2^*$};

\end{scope} 

\draw (12,0) node[scale=1] {$) - \varphi($};

\begin{scope}[xshift = 14.5cm, rotate=90]

\draw (0:1) arc (0:360:1);

\fill (60:1) circle(0.05);
\fill (120:1) circle(0.05);

\fill (240:1) circle(0.05);
\fill (300:1) circle(0.05);

\draw (60:1) to (240:1);
\draw (120:1) to (300:1);

\draw (60:1.3) node[scale=0.7] {$p_2$};
\draw (120:1.3) node[scale=0.7] {$p_3$};

\draw (240:1.4) node[scale=0.7] {$p_3^*$};
\draw (300:1.4) node[scale=0.7] {$p_2^*$};

\end{scope} 

\draw (16.2,0) node[scale=1] {$)$};

\draw (6.3,-2) node[scale=1] {$= (x-2)x^2+x^2-(x-1)x = (x-1)^2x,$};
\end{tikzpicture}
\end{center}
though the choice of $a = (p_2,p_2^*)$ or $a=(p_3,p_3^*)$ gives a quicker computation~:
\begin{center}
\begin{tikzpicture}[scale=0.5]

\draw (0:0) node[scale=1] {$\varphi(\mathcal{D}) = (x-1) \varphi($};

\begin{scope}[xshift = 4.3cm, rotate=90]

\draw (0:1) arc (0:360:1);

\fill (0:1) circle(0.05);

\fill (120:1) circle(0.05);
\fill (180:1) circle(0.05);
\fill (240:1) circle(0.05);

\draw (0:1) to (180:1);

\draw (120:1) to[bend left] (240:1);

\draw (0:1.3) node[scale=0.7] {$p_1$};

\draw (120:1.3) node[scale=0.7] {$p_3$};
\draw (180:1.3) node[scale=0.7] {$p_1^*$};
\draw (240:1.4) node[scale=0.7] {$p_3^*$};

\end{scope} 
\draw (5.6,0) node[scale=1] {$)$};

\draw (1.2,-2) node[scale=1] {$= (x-1)^2x.$};
\end{tikzpicture}
\end{center}

\begin{definition}
\label{defi:circulardiagrams}
For all $n \geq 1$, let $\mathcal{D}_n$ be the $n$-chord diagram where every chord intersects all the other chords, \textit{i.e.},

\begin{center}
\begin{tikzpicture}[scale=0.7]

\draw (0:0) node[scale=1] {$\mathcal{D}_n = ((i,n+i) : i \in [n]) =$};
\draw (6.2,0) node[scale=1] {$.$};

\begin{scope}[xshift=5cm,rotate=60]

\draw (0:1) arc (0:90:1) ;
\draw (90:1) [dashed] arc (90:180:1) ;
\draw (180:1) arc (180:270:1) ;
\draw (270:1) [dashed] arc (270:360:1) ;

\fill (0:1) circle(0.05);
\draw (0:1.3) node[scale=0.7] {$p_{n-1}^*$};
\fill (30:1) circle(0.05);
\draw (30:1.3) node[scale=0.7] {$p_n^*$};
\fill (60:1) circle(0.05);
\draw (60:1.3) node[scale=0.7] {$p_1$};
\fill (90:1) circle(0.05);
\draw (90:1.3) node[scale=0.7] {$p_2$};
\fill (180:1) circle(0.05);
\draw (180:1.3) node[scale=0.7] {$p_{n-1}$};
\fill (210:1) circle(0.05);
\draw (210:1.3) node[scale=0.7] {$p_n$};
\fill (240:1) circle(0.05);
\draw (240:1.3) node[scale=0.7] {$p_1^*$};
\fill (270:1) circle(0.05);
\draw (270:1.3) node[scale=0.7] {$p_2^*$};

\draw (0:1) -- (180:1);
\draw (30:1) -- (210:1);
\draw (60:1) -- (240:1);
\draw (90:1) -- (270:1);

 \fill[fill=gray, opacity=0.3]
    (0,0) -- (90:1) arc (90:180:1cm);
 \fill[fill=gray, opacity=0.3]
    (0,0) -- (270:1) arc (270:360:1cm);
\end{scope} 
\end{tikzpicture}
\end{center}
We set
$D_n = \varphi(\mathcal{D}_n) \in \mathbb{Z}[x].$
\end{definition}

The first elements of $(D_n)_{n \geq 1}$~:
\begin{align*}
D_1 &= x,\\
D_2 &= (x-1)x \equiv -x \mod x^2 ,\\
D_3 &= (x-2)(x-1)x \equiv 2x \mod x^2  ,\\
D_4 &= (x-3)(x-2)(x-1)x + x^3 - (x-1)^2x \equiv -7x \mod x^2  .
\end{align*}

The following is a conjecture from Lando that we will prove later (as a consequence of Theorem \ref{theo:mastertheo}).

\begin{conjecture}[Lando,2016]
\label{conjecture}
For all $n \geq 1$,
$$D_n \equiv (-1)^{n-1} h_{n-1} x \mod x^2  $$
where $(h_n)_{n \geq 0} = (1, 1, 2, 7, 38, 295, \hdots)$ is the sequence of the normalized median Genocchi numbers~\cite{h}, of which we give a reminder hereafter.
\end{conjecture}

\subsection{About the Genocchi numbers}
The Seidel triangle~\cite{DV} is a family of positive integers $(g_{i, j})_{1 \leq j \leq \lceil i/2 \rceil}$ (see Figure \ref{fig:seideltriangle}) defined by $g_{1,1} =~1$ and 
\begin{align*}
g_{2p, j} &= g_{2p-1, j} + g_{2p, j+1}, \\
g_{2p+1, j} &= g_{2p+1, j-1} + g_{2p, j}
\end{align*}
for all $p \geq 1$, where $g_{2p, p+1} = g_{2p+1,0} = 0$.
\begin{figure}[!h]
\begin{center}
\begin{tabular}{c|cccccccccc}
$\vdots$ & & & & & & & & & & $\iddots$\\
5 & & & & & & & & & 155 & $\hdots$\\
4 & & & & & & & 17 & 17 & 155 & $\hdots$\\
3 &   &   &   &   & 3  & 3 & 17  & 34  & 138 & $\hdots$\\
2 &   &   & 1 & 1 & 3 & 6 & 14 & 48 & 104 & $\hdots$\\
1 & 1 & 1 & 1 & 2 & 2 & 8 & 8 & 56 & 56 & $\hdots$ \\
\hline
$j/i$ & 1 & 2 & 3 & 4 & 5 & 6 & 7 & 8 & 9 & $\hdots$ 
\end{tabular}
\end{center}
\caption{The Seidel triangle.}
\label{fig:seideltriangle}
\end{figure}

The Genocchi numbers $(G_{2n})_{n \geq 1} = (1,1,3,17,155,2073,\hdots)$~\cite{G} and the median Genocchi numbers $(H_{2n+1})_{n \geq 0} = (1,2,8,56,608,\hdots)$~\cite{H} can be defined as the positive integers $G_{2n} = g_{2n-1,n}$ and $H_{2n+1} = g_{2n+2,1}$~\cite{DV}. It is well known that $H_{2n+1}$ is divisible by $2^n$ for all $n \geq 0$~\cite{BD}. The normalized median Genocchi numbers $(h_n)_{n \geq 0} = (1, 1, 2, 7, 38, 295, \hdots)$ are the positive integers defined by 
$$h_n = H_{2n+1}/2^n.$$

\begin{remark}
In view of Formula (\ref{eq:relationsbetweenvarphi}) with $\lambda = 2$, Conjecture \ref{conjecture} is equivalent to
$$\varphi_{\text{sl}_2}(\mathcal{D}_n) \equiv (-1)^{n-1} H_{2n-1}x \mod x^2$$
for all $n \geq 1$.
\end{remark}

There exist many combinatorial models of the different kinds of Genocchi numbers. Here, for all $n \geq 0$, we consider~:
\begin{itemize}
\item the set $PD2_n$ of the Dumont permutations of the second kind, that is, the permutations $\sigma \in \mathfrak{S}_{2n+2}$ such that $\sigma(2i-1) > 2i-1$ and $\sigma(2i) < 2i$ for all $i \in [n+1]$;
\item the subset $PD2N_n \subset PD2_n$ of the normalized such permutations, defined as the $\sigma \in PD2_n$ such that $\sigma^{-1}(2i) < \sigma^{-1}(2i+1)$ for all $i \in [n]$.
\end{itemize}
It is known that $H_{2n+1} = \# PD2_n$~\cite{Dumont} and $h_n = \# PD2N_n$~\cite{Kreweras,Feigin}.

Kreweras~\cite{Kreweras} refined the integers $h_n$ through the Kreweras triangle \linebreak $(h_{n,k})_{n \geq 1,k \in [n]}$ (see Figure \ref{fig:krewerastriangle}) defined by $h_{1,1} = 1$ and, for all $n \geq 2$ and $k \in [3,n]$, 
\begin{align}
\label{eq:definitionhnk}
h_{n,1} &= h_{n-1,1} + h_{n-1,2} + \hdots + h_{n-1,n-1}, \nonumber \\
h_{n,2} &= 2 h_{n,1} - h_{n-1,1},\\
h_{n,k} &= 2h_{n,k-1} - h_{n,k-2} - h_{n-1,k-1} - h_{n-1,k-2}.\nonumber
\end{align}
 
\begin{figure}[!h]
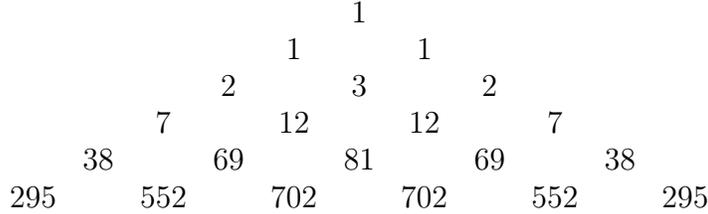

$$\begin{tabular}{ccccccccccccc}
& & & & & & 1 & & & & & \\
 && & & & 1 & & 1 & & & &  \\
 && & & 2 & & 3 & & 2 & & & \\
 & && 7 & & 12 & & 12 & & 7  \\
& & 38 & & 69 & & 81 & & 69 & & 38\\
& 295 & & 552 & & 702 & & 702 & & 552 & & 295
\end{tabular}$$
\caption{The first six lines of the Kreweras triangle.}
\label{fig:krewerastriangle}
\end{figure}

It is easy to see that for all $n \geq 0$, the set $PD2N_n$ has the partition $\{PD2N_{n,k}\}_{k \in [n]}$ where $PD2N_{n,k}$ is the set of the $\sigma \in PD2N_n$ such that $\sigma(1) = 2k$. Kreweras and Barraud~\cite{KB} proved that for all $n \geq~1$ and $k \in [n]$, the integer $h_{n,k}$ is the cardinality of $PD2N_{n,k}$. In particular, for all $n \geq 1$,
\begin{equation}
\label{eq:hn1}
h_{n,1} = \sum_{i=1}^{n-1} h_{n-1,i} = h_{n-1}.
\end{equation}

A visible property of the Kreweras triangle is the symmetry 
\begin{equation}
\label{eq:symetriehnk}
h_{n,k} = h_{n,n-k+1}
\end{equation}
for all $n \geq 1$ and $k \in [n]$. We can prove it combinatorially \cite{KB,Bigeni}, or by induction through the following easy formula
\begin{equation}
\label{eq:inductionhnk}
h_{n,k} - h_{n,k-1} = \sum_{i = k}^{n-1} h_{n-1,i} - \sum_{i=1}^{k-2} h_{n-1,i}
\end{equation}
for all $k \in [n]$ (where $h_{n,0}$ is defined as $0$), derived from Formulas (\ref{eq:definitionhnk}).

We now define a polynomial version of the Kreweras triangle, in the sense that it follows induction formulas analogous to Formulas (\ref{eq:definitionhnk}), which gives it anologous properties.

\subsection{A generalized Kreweras triangle}

Let $(K_{n,k})_{1 \leq k \leq n}$ be the family of polynomials defined by $K_{1,1} = x$ and, for all $n \geq 2$ and $k \in [3,n]$,
\begin{align}
\label{eq:definitionK}
K_{n,1} &= (x-1)K_{n-1,1} - K_{n-1,2} - \hdots - K_{n-1,n-1}, \nonumber \\
K_{n,2} &= 2 K_{n,1} + K_{n-1,1} - x^2 K_{n-2,1} \text{ (with $K_{0,1}$ defined as $1$),}\\
K_{n,k} &= 2K_{n,k-1} - K_{n,k-2} + K_{n-1,k-1} + K_{n-1,k-2} + 2xK_{n-2,k-2}.\nonumber
\end{align}

We depict in Figure \ref{fig:generalizedkrewerastriangle} the first lines of this triangle.

\begin{figure}[!h]

\begin{center}

\begin{tikzpicture}[scale=0.8]

\draw (-0.25,-0.25) rectangle (0.25,0.25);
\draw (-2.75,-1.25) rectangle (-1.25,-0.75);
\draw (1.25,-1.25) rectangle (2.75,-0.75);
\draw (-5.25,-2.25) rectangle (-2.75,-1.75);
\draw (-1.25,-2.25) rectangle (1.25,-1.75);
\draw (2.75,-2.25) rectangle (5.25,-1.75);
\draw (-7.8,-3.25) rectangle (-4.1,-2.75);
\draw (-4.1,-3.25) rectangle (-0,-2.75);
\draw (0,-3.25) rectangle (4.1,-2.75);
\draw (4.1,-3.25) rectangle (7.8,-2.75);

\draw (0,0) node[scale=0.75] {$x$};
\draw (-2,-1) node[scale=0.75] {$x^2-x$};
\draw (+2,-1) node[scale=0.75] {$x^2-x$};
\draw (-4,-2) node[scale=0.75] {$x^3-3x^2+2x$};
\draw (0,-2) node[scale=0.75] {$x^3-5x^2+3x$};
\draw (4,-2) node[scale=0.75] {$x^3-3x^2+2x$};
\draw (-5.95,-3) node[scale=0.75] {$x^4-6x^3+13x^2-7x$};
\draw (-2.05,-3) node[scale=0.75] {$x^4-10x^3+23x^2-12x$};
\draw (+2.05,-3) node[scale=0.75] {$x^4-10x^3+23x^2-12x$};
\draw (+5.95,-3) node[scale=0.75] {$x^4-6x^3+13x^2-7x$};
\end{tikzpicture}

\end{center}

\caption{The first four lines of $(K_{n,k})_{1 \leq k \leq n}$.}
\label{fig:generalizedkrewerastriangle}
\end{figure}

In view of Formulas (\ref{eq:definitionhnk}), it is easy to see that for all $n \geq 1$ and $k \in [n]$, the polynomial $K_{n,k}$ is of the kind 
\begin{equation}
\label{eq:formedeK}
K_{n,k} = x^n + \left( \sum_{i=1}^{n-2} (-1)^i a_{n,k,i} x^{n-i} \right) + (-1)^{n-1} h_{n,k} x
\end{equation}
for some positive integers $a_{n,k,1},\hdots,a_{n,k,n-2}$. The induction formulas of $(a_{n,k,1})_{1 \leq k \leq n}$ implied by Formulas (\ref{eq:definitionK}) can also be solved easily and give

\begin{equation}
\label{eq:an1}
a_{n,k,1} = \binom{n}{2} + 2(k-1)(n-k)
\end{equation}
for all $n \geq 1$ and $k \in [n]$. We can also note that
\begin{equation}
\label{eq:inductiona21}
a_{n,1,2} - a_{n-1,1,2} = \sum_{j = i}^{n-1} a_{n-1,i,1} = (n-1) \binom{n-1}{2} + 2 \binom{n-1}{3}
\end{equation}
in view of Formula (\ref{eq:an1}), which implies

\begin{equation*}
\label{eq:an21}
a_{n,1,2} = (n-2)(n-1)n(5n-7)/24
\end{equation*}
for all $n \geq 1$. This integer $a_{n,1,2}$ can be interpreted as the number of uplets $(w,x,z,t) \in [n-1]^4$ such that $w > x < y \geq z$. We show it by induction on~$n$. There are indeed $a_{1,1,2} = 0$ such uplets that belong to $[0]^4 = \emptyset$. Suppose now there are $a_{n-1,1,2}$ such uplets in $[n-2]^4$ for some $n \geq 2$, and let $(w,x,y,t)$ be such an uplet that belongs to $[n-1]^4 \backslash [n-2]^4$. There are three cases :
\begin{itemize}
\item either $(w,x,y,t) = (n-1,x,y,x)$ or $(n-1,x,y,y)$ for some $x < y$ (there are $2\binom{n-1}{2}$ such uplets);
\item or $(w,x,y,t) = (n-1,x,y,z)$ for some $x \neq z < y$ (there are $2\binom{n-1}{3}$ such uplets);
\item or $(w,x,y,t) = (w,x,n-1,z)$ for some $x < w \leq n-2$ and $z \in [n-1]$ (there are $(n-1)\binom{n-2}{2}$ such uplets),
\end{itemize}
for a total of $(n-1) \binom{n-1}{2}+2\binom{n-1}{3}$ elements, which gives the induction formula~(\ref{eq:inductiona21}).

The combinatorial interpretation of the coefficient $a_{n,k,i}$ in general (which would bridge the gap between $a_{n,k,1}$ and $h_{n,k}$) is an open problem.

The symmetry of the Kreweras triangle stated by Formula~(\ref{eq:symetriehnk}) has the extension
\begin{equation}
\label{eq:symetrieK}
K_{n,n+1-k} = K_{n,k},
\end{equation}
which can be proved by induction through the following formula which extends Formula (\ref{eq:inductionhnk}) and which we easily derive from Formulas (\ref{eq:definitionK})~: 
\begin{equation}
\label{eq:inductionK}
\begin{split}
K_{n,j}-K_{n,j-1} =& \sum_{i=1}^{j-2} K_{n-1,i} - \sum_{i=j}^{n-1} K_{n-1,i}\\
& + x \left( \sum_{i=1}^{j-2} K_{n-2,i} - \sum_{i=j-1}^{n-2} K_{n-2,i} \right)
\end{split}
\end{equation}
for all $j \in [n]$ (where $K_{n,0}$ is defined as $-xK_{n,1}$ for all $n \geq 0$; recall that $K_{0,1}$ has been defined as $1$).

\subsection{How the polynomial Kreweras triangle appears in the universal $\text{sl}_2$ weight system}

\begin{definition}
Let $n \geq 1$ and $k \in [0,n-1]$. We define two $n$-chords diagrams $\mathcal{A}_{n,k}$ and $\mathcal{B}_{n,k}$ as follows.

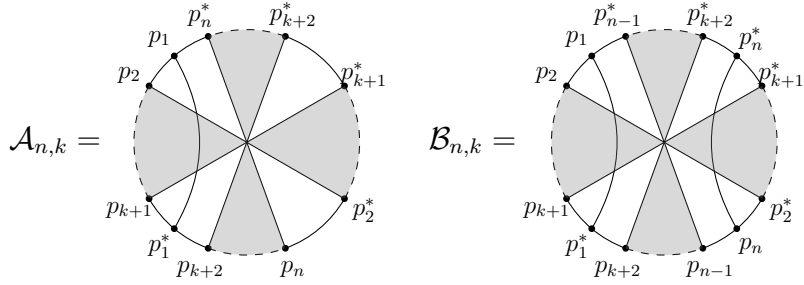
\begin{center}

\begin{tikzpicture}[scale=1.5]

\draw (-1.7,0) node[scale=1] {$\mathcal{A}_{n,k} = $};
\draw (2,0) node[scale=1] {$\mathcal{B}_{n,k} = $};

\begin{scope}[xshift= 3.7cm,rotate=110]

\draw (0:1) arc (0:20:1) ;
\draw (20:1) arc (20:40:1) ;
\draw (40:1) [dashed] arc (40:100:1) ;
\draw (100:1) arc (100:120:1) ;
\draw (120:1) arc (120:140:1) ;
\draw (140:1) [dashed] arc (140:180:1) ;
\draw (180:1) arc (180:220:1) ;
\draw (220:1) [dashed] arc (220:280:1) ;
\draw (280:1) arc (280:320:1) ;
\draw (320:1) [dashed] arc (320:360:1) ;

\fill (0:1) circle(0.03);
\draw (0:1.2) node[scale=0.8] {$p_{n-1}^*$};
\fill (20:1) circle(0.03);
\draw (20:1.2) node[scale=0.8] {$p_1$};
\fill (40:1) circle(0.03);
\draw (40:1.2) node[scale=0.8] {$p_2$};
\fill (100:1) circle(0.03);
\draw (100:1.2) node[scale=0.8] {$p_{k+1}$};
\fill (120:1) circle(0.03);
\draw (120:1.2) node[scale=0.8] {$p_1^*$};
\fill (140:1) circle(0.03);
\draw (140:1.2) node[scale=0.8] {$p_{k+2}$};
\fill (180:1) circle(0.03);
\draw (180:1.2) node[scale=0.8] {$p_{n-1}$};
\fill (220:1) circle(0.03);
\draw (220:1.2) node[scale=0.8] {$p_2^*$};
\fill (280:1) circle(0.03);
\draw (280:1.2) node[scale=0.8] {$p_{k+1}^*$};
\fill (320:1) circle(0.03);
\draw (320:1.2) node[scale=0.8] {$p_{k+2}^*$};
\fill (200:1) circle(0.03);
\draw (200:1.2) node[scale=0.8] {$p_{n}$};
\fill (300:1) circle(0.03);
\draw (300:1.2) node[scale=0.8] {$p_{n}^*$};

\draw (0:1) -- (180:1);
\draw (20:1) to[bend left] (120:1);
\draw (40:1) -- (220:1);
\draw (100:1) -- (280:1);
\draw (140:1) -- (320:1);
\draw (200:1) to[bend left] (300:1);

 \fill[fill=gray, opacity=0.3]
    (0,0) -- (40:1) arc (40:100:1cm);
 \fill[fill=gray, opacity=0.3]
    (0,0) -- (140:1) arc (140:180:1cm);
     \fill[fill=gray, opacity=0.3]
    (0,0) -- (220:1) arc (220:280:1cm);
     \fill[fill=gray, opacity=0.3]
    (0,0) -- (320:1) arc (320:360:1cm);
\end{scope} 

\begin{scope}[rotate=110]

\draw (0:1) arc (0:20:1) ;
\draw (20:1) arc (20:40:1) ;
\draw (40:1) [dashed] arc (40:100:1) ;
\draw (100:1) arc (100:120:1) ;
\draw (120:1) arc (120:140:1) ;
\draw (140:1) [dashed] arc (140:180:1) ;
\draw (180:1) arc (180:220:1) ;
\draw (220:1) [dashed] arc (220:280:1) ;
\draw (280:1) arc (280:320:1) ;
\draw (320:1) [dashed] arc (320:360:1) ;

\fill (0:1) circle(0.03);
\draw (0:1.2) node[scale=0.8] {$p_{n}^*$};
\fill (20:1) circle(0.03);
\draw (20:1.2) node[scale=0.8] {$p_1$};
\fill (40:1) circle(0.03);
\draw (40:1.2) node[scale=0.8] {$p_2$};
\fill (100:1) circle(0.03);
\draw (100:1.2) node[scale=0.8] {$p_{k+1}$};
\fill (120:1) circle(0.03);
\draw (120:1.2) node[scale=0.8] {$p_1^*$};
\fill (140:1) circle(0.03);
\draw (140:1.2) node[scale=0.8] {$p_{k+2}$};
\fill (180:1) circle(0.03);
\draw (180:1.2) node[scale=0.8] {$p_{n}$};
\fill (220:1) circle(0.03);
\draw (220:1.2) node[scale=0.8] {$p_2^*$};
\fill (280:1) circle(0.03);
\draw (280:1.2) node[scale=0.8] {$p_{k+1}^*$};
\fill (320:1) circle(0.03);
\draw (320:1.2) node[scale=0.8] {$p_{k+2}^*$};

\draw (0:1) -- (180:1);
\draw (20:1) to[bend left] (120:1);
\draw (40:1) -- (220:1);
\draw (100:1) -- (280:1);
\draw (140:1) -- (320:1);

 \fill[fill=gray, opacity=0.3]
    (0,0) -- (40:1) arc (40:100:1cm);
 \fill[fill=gray, opacity=0.3]
    (0,0) -- (140:1) arc (140:180:1cm);
     \fill[fill=gray, opacity=0.3]
    (0,0) -- (220:1) arc (220:280:1cm);
     \fill[fill=gray, opacity=0.3]
    (0,0) -- (320:1) arc (320:360:1cm);
\end{scope} 

\end{tikzpicture}

\end{center}

We then define two polynomials $A_{n,k} = \varphi(\mathcal{A}_{n,k})$ and $B_{n,k} = \varphi(\mathcal{B}_{n,k})$.
Note that :
\begin{itemize}
\item the chord $(p_1,p_1^*)$ of $\mathcal{A}_{n,k}$ or $\mathcal{B}_{n,k}$ (and the chord $(p_n,p_n^*)$ of $\mathcal{B}_{n,k}$) intersects exactly $k$ chords;
\item for all $n \geq 1$, $(A_{n,0},A_{n,1}) = (x D_{n-1},(x-1)D_{n-1})$ (where $D_0$ is defined as $1$), and $(B_{n,0},B_{n,1}) = (x^2 D_{n-2},(x-1)^2 D_{n-2})$ for all $n \geq 2$;
\item $A_{n,n-1} = B_{n,n-1} = D_n$ for all $n \geq 1$.
\end{itemize}
We also set $A_{n,-1} = 0$ and $B_{n,-1}=-xD_n$.
\end{definition}

\begin{remark}
\label{rem:DeltaequalsA}
For all $1 \leq i < j \leq n$, it is straightforward that
$$\Delta_{i,j}(\mathcal{D}_n) = B_{n,j-i-1} - B_{n,n-1-(j-i)}.$$
\end{remark}

The main result of this paper is the following, which implies Conjecture \ref{conjecture}.

\begin{theorem}
\label{theo:mastertheo}
For all $n \geq 1$ and $k \in [0,n-1]$, we have
\begin{align*}
A_{n,k-1} - A_{n,k} &= K_{n-1,k},\label{eq:A}\tag{$13_{n,k}$}\\
B_{n,k-1} - B_{n,n-k-1} &= K_{n,k} - K_{n,k+1},\label{eq:B}\tag{$14_{n,k}$}
\end{align*}
(recall that $K_{0,1} = 1$ and $K_{n,0} = -xK_{n,1}$ for all $n \geq 0$).
\end{theorem}
\setcounter{equation}{8}
Note that Formula ($13_{n+1,0}$) and Formula ($14_{n,0}$) are both equivalent to $D_n = K_{n,1}$ for all $n \geq 1$, which indeed proves Conjecture \ref{conjecture} in view of Formula (\ref{eq:formedeK}) and Formula (\ref{eq:hn1}).
\setcounter{equation}{14}

We prove Theorem \ref{theo:mastertheo} in Section \ref{sec:proof}.

In Section \ref{sec:open}, we discuss open problems related to it, among which a more general conjecture from Lando.

\section{Proof of Theorem \ref{theo:mastertheo}}
\label{sec:proof}
For $n = 1$ and $k = 0$, we have $A_{1,-1} - A_{1,0} = - x = K_{0,0}$ and $B_{1,-1} - B_{1,0} = -x^2 - x = K_{1,0} - K_{1,1}$, \textit{i.e.}, Theorem \ref{theo:mastertheo} is true for this case. Assume that it is true for some $n \geq 1$ and for all $k \in [0,n-1]$. In particular, Formula ($14_{n,0}$) being true implies $K_{n,0} = -xD_n$, hence Formula ($13_{n+1,0}$).

\begin{lemma}
\label{lem:complementariteA}
For all $k \in [n-1]$,
$$A_{n,k-1} + A_{n,n-k} = x D_{n-1} + D_n.$$
\end{lemma}
\begin{proof}
We have $A_{n,0} = x D_{n-1}$ and $A_{n,n-1} =  D_n$ so the equality is true for $k = 1$. Suppose it is for some $k \in [n-2]$. In view of Formula ($13_{n,k}$), Formula ($13_{n,n-k}$) and Formula (\ref{eq:symetrieK}), we have $A_{n,k-1} + A_{n,n-k} = A_{n,k} + A_{n,n-k-1}$, so the lemma is true by induction on $k$.
\end{proof}

\begin{lemma}
\label{lem:expansionBintermsofA}
For all $k \in [0,n]$,
$$A_{n+1,k-1} - A_{n+1,k} = -xD_n  -  \sum_{i=0}^{k-1} B_{n,i-1} - B_{n,n-i-1}  .$$
\end{lemma}

\begin{proof}
For all $l \in [0,n]$, from Definition \ref{defi:varphi} (with $\mathcal{D}= \mathcal{A}_{n+1,l}$ and $a = (p_1,p_1^*)$) and Remark~\ref{rem:DeltaequalsA}, we have
\begin{equation}
\label{eq:c}
A_{n+1,l} = (x-l)D_{n} + \sum_{2 \leq i < j \leq l+1} \Delta_{i,j}(\mathcal{D}_n)
\end{equation}
from which the lemma follows in view of Remark \ref{rem:DeltaequalsA} and $B_{n,-1}- B_{n,n-1} = -(x+1)D_n$.
\end{proof}

Now, Lemma \ref{lem:expansionBintermsofA}, Formula \hyperref[eq:B]{$(14_{n,k})$} for all $k \in [0,n-1]$, and the equality $-xD_n = K_{n,0}$ imply Formula \hyperref[eq:A]{$(13_{n+1,k})$} for all $k \in [0,n]$. Afterwards, from Formulas (\ref{eq:definitionK}) we have
$$K_{n+1,1} = x K_{n,1} - \sum_{k = 1}^n K_{n,k},$$
which, in view of $x K_{n,1} = x D_n = A_{n+1,0}$ and Formula ($13_{n+1,k}$) for all $k \in [n]$, becomes
$$K_{n+1,1} = A_{n+1,n} = D_{n+1}.$$
This proves Formula~\hyperref[eq:B]{$(14_{n+1,0})$}. It remains to prove Formula \hyperref[eq:B]{$(14_{n+1,k})$} for all $k \in [n]$.

\begin{definition}
For all $n$-chord diagram $\mathcal{D}$ and for all quadruplet of integers $(a,b,c,d)$ such that $1 \leq a \leq b < c \leq d \leq n$, we define two polynomials
\begin{align*}
T_{a,b}(\mathcal{D}) &= \sum_{a \leq s < t \leq b} \Delta_{s,t}(\mathcal{D}),\\
R_{a,b,c,d}(\mathcal{D}) &= \sum_{s = a}^b \sum_{t = c}^d \Delta_{s,t}(\mathcal{D}). 
\end{align*} 
They are related by the equality
\begin{equation}
\label{eq:relationTR}
T_{a,c}(\mathcal{D}) = T_{a,b}(\mathcal{D}) + T_{b+1,c}(\mathcal{D}) + R_{a,b,b+1,c}(\mathcal{D}).
\end{equation}
\end{definition}

\begin{lemma}
\label{lem:expands}
For all $l \in [0,n-1]$,
\begin{align}
B_{n+1,l} &= (x-l) A_{n,l} + T_{2,l+1}(\mathcal{A}_{n,l}) ,\label{eq:a}\\
A_{n+1,l} &= (x-(n-1)) A_{n,l} + T_{2,n}(\mathcal{A}_{n,l}) ,\label{eq:b}.
\end{align}
\end{lemma}

\begin{proof}
By applying Definition \ref{defi:varphi} on $\mathcal{D} = \mathcal{B}_{n+1,l}$ (respectively on $\mathcal{D} = \mathcal{A}_{n+1,l}$) with $a = (p_{n+1},p_{n+1}^*)$, we obtain Formula (\ref{eq:a}) (respectively Formula (\ref{eq:b})).
\end{proof}

\begin{lemma}
\label{lem:simpl1}
The polynomial $A_{n+1,n-k} - B_{n+1,n-k}+xA_{n,n-k}$ equals
$$(x-(k-1))A_{n,n-k}+ \sum_{j = n-k+2}^{n} \sum_{i = 2}^{j-1} \Delta_{i,j}(\mathcal{A}_{n,n-k})$$
for all $k \in [n]$.
\end{lemma}

\begin{proof}
It is an application of Formula (\ref{eq:a}) and Formula (\ref{eq:b}) with $l = n-k$.
\end{proof}

\begin{lemma}
\label{lem:simpl2}
The polynomial $A_{n+1,k-1} - B_{n+1,k-1}+xA_{n,k-1}$ equals
$$(x-(k-1))A_{n,n-k}+  \sum_{2 \leq i < j \leq k} \Delta_{i,j}(\mathcal{D}_n) + x \Delta_{i,j}(\mathcal{D}_{n-1})- \Delta_{i,j}(\mathcal{A}_{n,k-1})$$
for all $k \in [n]$.
\end{lemma}

\begin{proof}
By applying Formula (\ref{eq:a}) and Formula (\ref{eq:c}) with $l = k-1$, and in view of Lemma \ref{lem:complementariteA}, we obtain that $A_{n+1,k-1}-B_{n+1,k-1}+xA_{n,k-1}$ equals
\begin{equation*}
\begin{split}
(x-(k-1))A_{n,n-k}+ x(A_{n,k-1}-(x-(k-1))D_{n-1})\\+ \sum_{2 \leq i < j \leq k} \Delta_{i,j}(\mathcal{D}_n) - \Delta_{i,j}(\mathcal{A}_{n,k-1}).
\end{split}
\end{equation*}
The lemma then follows by applying Formula (\ref{eq:c}) one last time to the expression $A_{n,k-1}-(x-(k-1))D_{n-1}$.
\end{proof}

As we will see at the end of this section, the rest of the proof is to show that the polynomials in Lemma \ref{lem:simpl1} and Lemma \ref{lem:simpl2} are equal, in other words, to obtain
\begin{equation}
\label{eq:asimplifier1}
\begin{split}
\sum_{j = n-k+2}^{n} \sum_{i = 2}^{j-1} &\Delta_{i,j}(\mathcal{A}_{n,n-k}) = \\ & \sum_{2 \leq i < j \leq k} \Delta_{i,j}(\mathcal{D}_n) + x \Delta_{i,j}(\mathcal{D}_{n-1}) - \Delta_{i,j}(\mathcal{A}_{n,k-1}).
\end{split}
\end{equation}

\begin{lemma}
\label{lem:lescontributionsnulles}
For all $k \in [n]$, $j \in [n-k+2,n]$ and $i \in [2,n-k+1]$,
$$\Delta_{i,j}(\mathcal{A}_{n,n-k}) + \Delta_{n-k+3-i,2n-k+2-j}(\mathcal{A}_{n,n-k})=0.$$
\end{lemma}

\begin{proof}
$\Delta_{i,j}(\mathcal{A}_{n,n-k}) = \varphi(\mathcal{D}^1_{n,i,j}) - \varphi(\mathcal{D}^2_{n,i,j})$ where
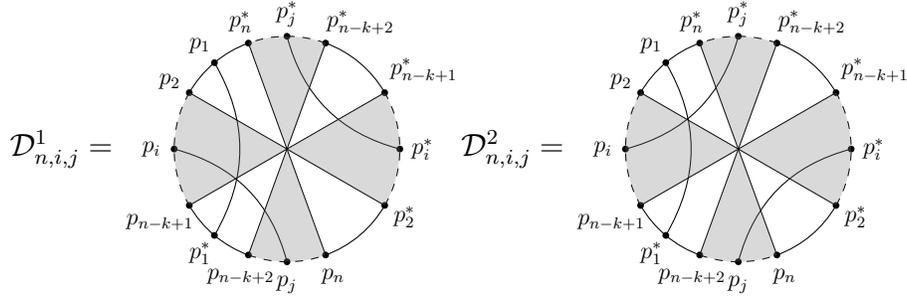
\begin{center}
\begin{tikzpicture}[scale=1.5]

\draw (-2,0) node[scale=1] {$\mathcal{D}^1_{n,i,j} = $};
\draw (2,0) node[scale=1] {$\mathcal{D}^2_{n,i,j} = $};

\begin{scope}[xshift= 4cm,rotate=110]

\draw (0:1) arc (0:20:1) ;
\draw (20:1) arc (20:40:1) ;
\draw (40:1) [dashed] arc (40:100:1) ;
\draw (100:1) arc (100:120:1) ;
\draw (120:1) arc (120:140:1) ;
\draw (140:1) [dashed] arc (140:180:1) ;
\draw (180:1) arc (180:220:1) ;
\draw (220:1) [dashed] arc (220:280:1) ;
\draw (280:1) arc (280:320:1) ;
\draw (320:1) [dashed] arc (320:360:1) ;

\fill (0:1) circle(0.03);
\draw (0:1.2) node[scale=0.75] {$p_{n}^*$};
\fill (20:1) circle(0.03);
\draw (20:1.2) node[scale=0.75] {$p_1$};
\fill (40:1) circle(0.03);
\draw (40:1.2) node[scale=0.75] {$p_2$};
\fill (70:1) circle(0.03);
\draw (70:1.2) node[scale=0.75] {$p_i$};
\fill (100:1) circle(0.03);
\draw (100:1.3) node[scale=0.75] {$p_{n-k+1}$};
\fill (120:1) circle(0.03);
\draw (120:1.2) node[scale=0.75] {$p_1^*$};
\fill (140:1) circle(0.03);
\draw (140:1.2) node[scale=0.75] {$p_{n-k+2}$};
\fill (160:1) circle(0.03);
\draw (160:1.2) node[scale=0.75] {$p_j$};
\fill (180:1) circle(0.03);
\draw (180:1.2) node[scale=0.75] {$p_{n}$};
\fill (220:1) circle(0.03);
\draw (220:1.2) node[scale=0.75] {$p_2^*$};
\fill (280:1) circle(0.03);
\draw (280:1.4) node[scale=0.75] {$p_{n-k+1}^*$};
\fill (320:1) circle(0.03);
\draw (310:1.3) node[scale=0.75] {$p_{n-k+2}^*$};
\fill (250:1) circle(0.03);
\draw (250:1.2) node[scale=0.75] {$p_{i}^*$};
\fill (340:1) circle(0.03);
\draw (340:1.2) node[scale=0.75] {$p_{j}^*$};

\draw (0:1) -- (180:1);
\draw (20:1) to[bend left] (120:1);
\draw (40:1) -- (220:1);
\draw (100:1) -- (280:1);
\draw (140:1) -- (320:1);
\draw (70:1) to[bend right] (340:1);
\draw (160:1) to[bend left] (250:1);

 \fill[fill=gray, opacity=0.3]
    (0,0) -- (40:1) arc (40:100:1cm);
 \fill[fill=gray, opacity=0.3]
    (0,0) -- (140:1) arc (140:180:1cm);
     \fill[fill=gray, opacity=0.3]
    (0,0) -- (220:1) arc (220:280:1cm);
     \fill[fill=gray, opacity=0.3]
    (0,0) -- (320:1) arc (320:360:1cm);
\end{scope} 

\begin{scope}[rotate=110]

\draw (0:1) arc (0:20:1) ;
\draw (20:1) arc (20:40:1) ;
\draw (40:1) [dashed] arc (40:100:1) ;
\draw (100:1) arc (100:120:1) ;
\draw (120:1) arc (120:140:1) ;
\draw (140:1) [dashed] arc (140:180:1) ;
\draw (180:1) arc (180:220:1) ;
\draw (220:1) [dashed] arc (220:280:1) ;
\draw (280:1) arc (280:320:1) ;
\draw (320:1) [dashed] arc (320:360:1) ;

\fill (0:1) circle(0.03);
\draw (0:1.2) node[scale=0.75] {$p_{n}^*$};
\fill (20:1) circle(0.03);
\draw (20:1.2) node[scale=0.75] {$p_1$};
\fill (40:1) circle(0.03);
\draw (40:1.2) node[scale=0.75] {$p_2$};
\fill (70:1) circle(0.03);
\draw (70:1.2) node[scale=0.75] {$p_i$};
\fill (100:1) circle(0.03);
\draw (100:1.3) node[scale=0.75] {$p_{n-k+1}$};
\fill (120:1) circle(0.03);
\draw (120:1.2) node[scale=0.75] {$p_1^*$};
\fill (140:1) circle(0.03);
\draw (140:1.2) node[scale=0.75] {$p_{n-k+2}$};
\fill (160:1) circle(0.03);
\draw (160:1.2) node[scale=0.75] {$p_j$};
\fill (180:1) circle(0.03);
\draw (180:1.2) node[scale=0.75] {$p_{n}$};
\fill (220:1) circle(0.03);
\draw (220:1.2) node[scale=0.75] {$p_2^*$};
\fill (280:1) circle(0.03);
\draw (280:1.4) node[scale=0.75] {$p_{n-k+1}^*$};
\fill (320:1) circle(0.03);
\draw (310:1.3) node[scale=0.75] {$p_{n-k+2}^*$};
\fill (250:1) circle(0.03);
\draw (250:1.2) node[scale=0.75] {$p_{i}^*$};
\fill (340:1) circle(0.03);
\draw (340:1.2) node[scale=0.75] {$p_{j}^*$};

\draw (0:1) -- (180:1);
\draw (20:1) to[bend left] (120:1);
\draw (40:1) -- (220:1);
\draw (100:1) -- (280:1);
\draw (140:1) -- (320:1);
\draw (70:1) to[bend left] (160:1);
\draw (250:1) to[bend left] (340:1);

 \fill[fill=gray, opacity=0.3]
    (0,0) -- (40:1) arc (40:100:1cm);
 \fill[fill=gray, opacity=0.3]
    (0,0) -- (140:1) arc (140:180:1cm);
     \fill[fill=gray, opacity=0.3]
    (0,0) -- (220:1) arc (220:280:1cm);
     \fill[fill=gray, opacity=0.3]
    (0,0) -- (320:1) arc (320:360:1cm);
\end{scope} 

\end{tikzpicture}

\end{center}

Now, if $\Sigma$ is the axial symmetry that maps $p_1$ to $p_1^*$, it is easy to check that
$$\Sigma(\mathcal{D}^1_{n,i,j}) = \mathcal{D}^2_{n,n-k+3-i,2n-k+2-j}.$$
Moreover, from Definition \ref{defi:varphi}, it is straightforward by induction on the order $n$ of any chord diagram $\mathcal{D}$ that $\varphi(\Sigma(\mathcal{D})) = \varphi(\mathcal{D})$, thence the lemma.
\end{proof}

In view of Lemma \ref{lem:lescontributionsnulles}, Formula (\ref{eq:asimplifier1}) that we need to prove becomes

\begin{equation}
\label{eq:asimplifier2}
\sum_{2 \leq i < j \leq k} \begin{pmatrix}
\Delta_{n-k+i,n-k+j}(\mathcal{A}_{n,n-k}) + \Delta_{i,j}(\mathcal{A}_{n,k-1}) \\
- \Delta_{i,j}(\mathcal{D}_n) - x \Delta_{i,j}(\mathcal{D}_{n-1})
\end{pmatrix} = 0.
\end{equation}

\begin{lemma}
\label{lem:expansionB}
For all $2 \leq i < j \leq k \leq n$,
\begin{align*}
B_{n,j-i-1} = &(x-(n-3))B_{n-1,j-i-1}\\&+ T_{2,n-2}(\mathcal{B}_{n-1,j-i-1})\\&- 2 R_{2,k-i,k-i+1,n-2}(\mathcal{B}_{n-1,j-i-1})\\
B_{n,n-1-(j-i)} = &(x-(n-1))B_{n-1,n-2-(j-i)} \\
&+ T_{1,n-1}(\mathcal{B}_{n-1,n-2-(j-i)})\\&- 2 R_{1,k-j+1,k-j+2,n-1}(\mathcal{B}_{n-1,n-2-(j-i)}).
\end{align*}
Incidentally, the sums of polynomials $R_{2,k-i,k-i+1,n-2}(\mathcal{B}_{n-1,j-i-1})$ and $R_{1,k-j+1,k-j+2,n-1}(\mathcal{B}_{n-1,n-2-(j-i)})$ do not depend on $k$.
\end{lemma}

\begin{proof}
By applying Definition \ref{defi:varphi} on $\mathcal{D} = \mathcal{B}_{n,j-i-1}$ and the chord $a = (p_{k-i+1},p_{k-i+1}^*)$ (respectively on $\mathcal{D} = \mathcal{B}_{n,n-1-(j-i)}$ and the chord $a = (p_{k-j+2},p_{k-j+2}^*)$), we obtain the two respective formulas
\begin{align*}
B_{n,j-i-1} = &(x-(n-3))B_{n-1,j-i-1} \\
&+ T_{2,k-i}(\mathcal{B}_{n-1,j-i-1}) + T_{k-i+1,n-2}(\mathcal{B}_{n-1,j-i-1})\\&- R_{2,k-i,n-2}(\mathcal{B}_{n-1,j-i-1}),\\
B_{n,n-1-(j-i)} = &(x-(n-1))B_{n-1,n-2-(j-i)} \\
&+ T_{1,k-j+1}(\mathcal{B}_{n-1,n-2-(j-i)}) + T_{k-j+2,n-1}(\mathcal{B}_{n-1,n-2-(j-i)})\\&- R_{1,k-j+1,n-1}(\mathcal{B}_{n-1,n-2-(j-i)}),
\end{align*}
and the equations of the lemma then follow from Formula (\ref{eq:relationTR}).\end{proof}

\begin{lemma}
\label{lem:lescontributionsnonnulles}
Let $k \in [n]$ and $2 \leq i < j \leq k$. We consider the pair $$(I,J) = (n-k+i,n-k+j).$$ Then :
$$\Delta_{I,J}(\mathcal{A}_{n,n-k})+ \Delta_{i,j}(\mathcal{A}_{n,k-1}) = \Delta_{i,j}(\mathcal{D}_n) + x  \Delta_{i,j}(\mathcal{D}_{n-1}) .$$
\end{lemma}

\begin{proof}We have 
\begin{align*}
\Delta_{i,j}(\mathcal{A}_{n,k-1}) &= \varphi(\mathcal{D}^1_{n,i,j}) - \varphi(\mathcal{D}^2_{n,i,j}),\\
\Delta_{I,J}(\mathcal{A}_{n,n-k}) &= \varphi(\mathcal{D}^3_{n,I,J}) - \varphi(\mathcal{D}^4_{n,I,J})
\end{align*}
where
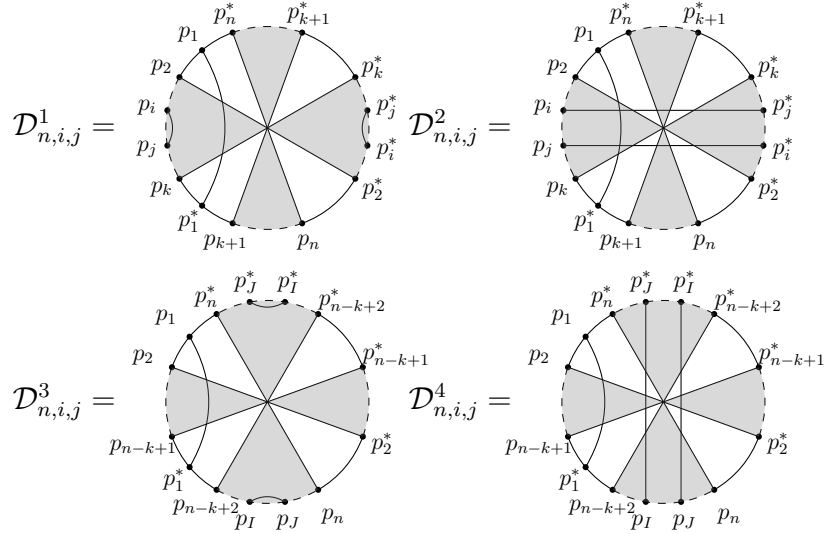
\begin{center}
\begin{tikzpicture}[scale=1.35]

\draw (-2,0) node[scale=1] {$\mathcal{D}^1_{n,i,j} = $};
\draw (1.9,0) node[scale=1] {$\mathcal{D}^2_{n,i,j} = $};
\draw (-2,-2.7) node[scale=1] {$\mathcal{D}^3_{n,i,j} = $};
\draw (1.9,-2.7) node[scale=1] {$\mathcal{D}^4_{n,i,j} = $};

\begin{scope}[xshift=3.9cm,yshift = -2.7cm, rotate=200]

\draw (0:1) arc (0:20:1) ;
\draw (20:1) arc (20:40:1) ;
\draw (40:1) [dashed] arc (40:100:1) ;
\draw (100:1) arc (100:120:1) ;
\draw (120:1) arc (120:140:1) ;
\draw (140:1) [dashed] arc (140:180:1) ;
\draw (180:1) arc (180:220:1) ;
\draw (220:1) [dashed] arc (220:280:1) ;
\draw (280:1) arc (280:320:1) ;
\draw (320:1) [dashed] arc (320:360:1) ;

\fill (0:1) circle(0.03);
\draw (0:1.3) node[scale=0.75] {$p_{n-k+1}$};
\fill (20:1) circle(0.03);
\draw (20:1.2) node[scale=0.75] {$p_1^*$};
\fill (40:1) circle(0.03);
\draw (40:1.2) node[scale=0.75] {$p_{n-k+2}$};
\fill (100:1) circle(0.03);
\draw (100:1.3) node[scale=0.75] {$p_{n}$};
\fill (140:1) circle(0.03);
\draw (140:1.2) node[scale=0.75] {$p_{2}^*$};
\fill (180:1) circle(0.03);
\draw (180:1.35) node[scale=0.75] {$p_{n-k+1}^*$};
\fill (220:1) circle(0.03);
\draw (210:1.3) node[scale=0.75] {$p_{n-k+2}^*$};
\fill (280:1) circle(0.03);
\draw (280:1.2) node[scale=0.75] {$p_{n}^*$};
\fill (-60:1) circle(0.03);
\draw (-60:1.3) node[scale=0.75] {$p_1$};
\fill (-40:1) circle(0.03);
\draw (-40:1.3) node[scale=0.75] {$p_2$};
\fill (60:1) circle(0.03);
\draw (60:1.2) node[scale=0.75] {$p_I$};
\fill (80:1) circle(0.03);
\draw (80:1.2) node[scale=0.75] {$p_J$};
\fill (240:1) circle(0.03);
\draw (240:1.2) node[scale=0.75] {$p_I^*$};
\fill (260:1) circle(0.03);
\draw (260:1.2) node[scale=0.75] {$p_J^*$};

\draw (0:1) -- (180:1);
\draw (-60:1) to[bend left] (20:1);
\draw (40:1) -- (220:1);
\draw (100:1) -- (280:1);
\draw (140:1) -- (320:1);
\draw (60:1) to (260:1);
\draw (80:1) to (240:1);

 \fill[fill=gray, opacity=0.3]
    (0,0) -- (40:1) arc (40:100:1cm);
 \fill[fill=gray, opacity=0.3]
    (0,0) -- (140:1) arc (140:180:1cm);
     \fill[fill=gray, opacity=0.3]
    (0,0) -- (220:1) arc (220:280:1cm);
     \fill[fill=gray, opacity=0.3]
    (0,0) -- (320:1) arc (320:360:1cm);
\end{scope} 

\begin{scope}[yshift = -2.7cm, rotate=200]

\draw (0:1) arc (0:20:1) ;
\draw (20:1) arc (20:40:1) ;
\draw (40:1) [dashed] arc (40:100:1) ;
\draw (100:1) arc (100:120:1) ;
\draw (120:1) arc (120:140:1) ;
\draw (140:1) [dashed] arc (140:180:1) ;
\draw (180:1) arc (180:220:1) ;
\draw (220:1) [dashed] arc (220:280:1) ;
\draw (280:1) arc (280:320:1) ;
\draw (320:1) [dashed] arc (320:360:1) ;

\fill (0:1) circle(0.03);
\draw (0:1.3) node[scale=0.75] {$p_{n-k+1}$};
\fill (20:1) circle(0.03);
\draw (20:1.2) node[scale=0.75] {$p_1^*$};
\fill (40:1) circle(0.03);
\draw (40:1.2) node[scale=0.75] {$p_{n-k+2}$};
\fill (100:1) circle(0.03);
\draw (100:1.3) node[scale=0.75] {$p_{n}$};
\fill (140:1) circle(0.03);
\draw (140:1.2) node[scale=0.75] {$p_{2}^*$};
\fill (180:1) circle(0.03);
\draw (180:1.35) node[scale=0.75] {$p_{n-k+1}^*$};
\fill (220:1) circle(0.03);
\draw (210:1.3) node[scale=0.75] {$p_{n-k+2}^*$};
\fill (280:1) circle(0.03);
\draw (280:1.2) node[scale=0.75] {$p_{n}^*$};
\fill (-60:1) circle(0.03);
\draw (-60:1.3) node[scale=0.75] {$p_1$};
\fill (-40:1) circle(0.03);
\draw (-40:1.3) node[scale=0.75] {$p_2$};
\fill (60:1) circle(0.03);
\draw (60:1.2) node[scale=0.75] {$p_I$};
\fill (80:1) circle(0.03);
\draw (80:1.2) node[scale=0.75] {$p_J$};
\fill (240:1) circle(0.03);
\draw (240:1.2) node[scale=0.75] {$p_I^*$};
\fill (260:1) circle(0.03);
\draw (260:1.2) node[scale=0.75] {$p_J^*$};

\draw (0:1) -- (180:1);
\draw (-60:1) to[bend left] (20:1);
\draw (40:1) -- (220:1);
\draw (100:1) -- (280:1);
\draw (140:1) -- (320:1);
\draw (60:1) to[bend left] (80:1);
\draw (240:1) to[bend left] (260:1);

 \fill[fill=gray, opacity=0.3]
    (0,0) -- (40:1) arc (40:100:1cm);
 \fill[fill=gray, opacity=0.3]
    (0,0) -- (140:1) arc (140:180:1cm);
     \fill[fill=gray, opacity=0.3]
    (0,0) -- (220:1) arc (220:280:1cm);
     \fill[fill=gray, opacity=0.3]
    (0,0) -- (320:1) arc (320:360:1cm);
\end{scope} 

\begin{scope}[rotate=110]

\draw (0:1) arc (0:20:1) ;
\draw (20:1) arc (20:40:1) ;
\draw (40:1) [dashed] arc (40:100:1) ;
\draw (100:1) arc (100:120:1) ;
\draw (120:1) arc (120:140:1) ;
\draw (140:1) [dashed] arc (140:180:1) ;
\draw (180:1) arc (180:220:1) ;
\draw (220:1) [dashed] arc (220:280:1) ;
\draw (280:1) arc (280:320:1) ;
\draw (320:1) [dashed] arc (320:360:1) ;

\fill (0:1) circle(0.03);
\draw (0:1.2) node[scale=0.75] {$p_{n}^*$};
\fill (20:1) circle(0.03);
\draw (20:1.2) node[scale=0.75] {$p_1$};
\fill (40:1) circle(0.03);
\draw (40:1.2) node[scale=0.75] {$p_2$};
\fill (100:1) circle(0.03);
\draw (100:1.2) node[scale=0.75] {$p_{k}$};
\fill (120:1) circle(0.03);
\draw (120:1.2) node[scale=0.75] {$p_1^*$};
\fill (140:1) circle(0.03);
\draw (140:1.2) node[scale=0.75] {$p_{k+1}$};
\fill (180:1) circle(0.03);
\draw (180:1.2) node[scale=0.75] {$p_{n}$};
\fill (220:1) circle(0.03);
\draw (220:1.2) node[scale=0.75] {$p_2^*$};
\fill (280:1) circle(0.03);
\draw (280:1.2) node[scale=0.75] {$p_{k}^*$};
\fill (320:1) circle(0.03);
\draw (320:1.2) node[scale=0.75] {$p_{k+1}^*$};
\fill (60:1) circle(0.03);
\draw (60:1.2) node[scale=0.75] {$p_i$};
\fill (80:1) circle(0.03);
\draw (80:1.2) node[scale=0.75] {$p_j$};
\fill (240:1) circle(0.03);
\draw (240:1.2) node[scale=0.75] {$p_{i}^*$};
\fill (260:1) circle(0.03);
\draw (260:1.2) node[scale=0.75] {$p_{j}^*$};

\draw (0:1) -- (180:1);
\draw (20:1) to[bend left] (120:1);
\draw (40:1) -- (220:1);
\draw (100:1) -- (280:1);
\draw (140:1) -- (320:1);
\draw (60:1) to[bend left] (80:1);
\draw (240:1) to[bend left] (260:1);

 \fill[fill=gray, opacity=0.3]
    (0,0) -- (40:1) arc (40:100:1cm);
 \fill[fill=gray, opacity=0.3]
    (0,0) -- (140:1) arc (140:180:1cm);
     \fill[fill=gray, opacity=0.3]
    (0,0) -- (220:1) arc (220:280:1cm);
     \fill[fill=gray, opacity=0.3]
    (0,0) -- (320:1) arc (320:360:1cm);
\end{scope} 

\begin{scope}[xshift=3.9cm, rotate=110]

\draw (0:1) arc (0:20:1) ;
\draw (20:1) arc (20:40:1) ;
\draw (40:1) [dashed] arc (40:100:1) ;
\draw (100:1) arc (100:120:1) ;
\draw (120:1) arc (120:140:1) ;
\draw (140:1) [dashed] arc (140:180:1) ;
\draw (180:1) arc (180:220:1) ;
\draw (220:1) [dashed] arc (220:280:1) ;
\draw (280:1) arc (280:320:1) ;
\draw (320:1) [dashed] arc (320:360:1) ;

\fill (0:1) circle(0.03);
\draw (0:1.2) node[scale=0.75] {$p_{n}^*$};
\fill (20:1) circle(0.03);
\draw (20:1.2) node[scale=0.75] {$p_1$};
\fill (40:1) circle(0.03);
\draw (40:1.2) node[scale=0.75] {$p_2$};
\fill (100:1) circle(0.03);
\draw (100:1.2) node[scale=0.75] {$p_{k}$};
\fill (120:1) circle(0.03);
\draw (120:1.2) node[scale=0.75] {$p_1^*$};
\fill (140:1) circle(0.03);
\draw (140:1.2) node[scale=0.75] {$p_{k+1}$};
\fill (180:1) circle(0.03);
\draw (180:1.2) node[scale=0.75] {$p_{n}$};
\fill (220:1) circle(0.03);
\draw (220:1.2) node[scale=0.75] {$p_2^*$};
\fill (280:1) circle(0.03);
\draw (280:1.2) node[scale=0.75] {$p_{k}^*$};
\fill (320:1) circle(0.03);
\draw (320:1.2) node[scale=0.75] {$p_{k+1}^*$};
\fill (60:1) circle(0.03);
\draw (60:1.2) node[scale=0.75] {$p_i$};
\fill (80:1) circle(0.03);
\draw (80:1.2) node[scale=0.75] {$p_j$};
\fill (240:1) circle(0.03);
\draw (240:1.2) node[scale=0.75] {$p_{i}^*$};
\fill (260:1) circle(0.03);
\draw (260:1.2) node[scale=0.75] {$p_{j}^*$};

\draw (0:1) -- (180:1);
\draw (20:1) to[bend left] (120:1);
\draw (40:1) -- (220:1);
\draw (100:1) -- (280:1);
\draw (140:1) -- (320:1);
\draw (80:1) to (240:1);
\draw (60:1) to (260:1);

 \fill[fill=gray, opacity=0.3]
    (0,0) -- (40:1) arc (40:100:1cm);
 \fill[fill=gray, opacity=0.3]
    (0,0) -- (140:1) arc (140:180:1cm);
     \fill[fill=gray, opacity=0.3]
    (0,0) -- (220:1) arc (220:280:1cm);
     \fill[fill=gray, opacity=0.3]
    (0,0) -- (320:1) arc (320:360:1cm);
\end{scope} 

\end{tikzpicture}

\end{center}

By applying Definition \ref{defi:varphi} with $a = (p_1,p_1^*)$, we obtain
\begin{align*}
\varphi(\mathcal{D}^1_{n,i,j}) = &(x-(k-3))B_{n-1,j-i-1} \\&+ T_{2,k-i}(\mathcal{B}_{n-1,j-i-1}) + T_{n-i+1,n-2}(\mathcal{B}_{n-1,j-i-1}) \\&- R_{2,k-i,n-i+1,n-2}(\mathcal{B}_{n-1,j-i-1}),\\
\varphi(\mathcal{D}^2_{n,i,j}) = &(x-(k-1))B_{n-1,n-2-(j-i)} \\&+ T_{1,k-j+1}(\mathcal{B}_{n-1,n-2-(j-i)}) + T_{n-j+2,n-1}(\mathcal{B}_{n-1,n-2-(j-i)}) \\&- R_{1,k-j+1,n-j+2,n-1}(\mathcal{B}_{n-1,n-2-(j-i)}),\\
\varphi(\mathcal{D}^3_{n,i,j}) = &(x-(n-k))B_{n-1,j-i-1} + T_{k-i+1,n-i}(\mathcal{B}_{n-1,j-i-1}),\\
\varphi(\mathcal{D}^4_{n,i,j}) = &(x-(n-k))B_{n-1,n-2-(j-i)} + T_{k-j+2,n-j+1}(\mathcal{B}_{n-1,n-2-(j-i)}).
\end{align*}
It is then a consequence of Formula (\ref{eq:relationTR}) and Lemma \ref{lem:expansionB} that 
\begin{align*}
\varphi(\mathcal{D}^1_{n,i,j}) + \varphi(\mathcal{D}^3_{n,i,j}) &= B_{n,j-i-1} + x B_{n-1,j-i-1},\\
\varphi(\mathcal{D}^2_{n,i,j}) + \varphi(\mathcal{D}^4_{n,i,j}) &= B_{n,n-1-(j-i)} + x B_{n-1,n-2-(j-i)},
\end{align*}
in view of
\begin{align*}
R_{2,k-i,k-i+1,n-i} + R_{2,k-i,n-i+1,n-2}  &= R_{2,k-i,k-i+1,n-2},\\
R_{1,k-j+1,k-j+2,n-j+1} + R_{1,k-j+1,n-j+2,n-1} &= R_{1,k-j+1,k-j+2,n-1}.
\end{align*}
This proves the lemma because of Remark \ref{rem:DeltaequalsA}.
\end{proof}

Lemma \ref{lem:lescontributionsnonnulles} proves Formula (\ref{eq:asimplifier2}). In other words, the results from Lemma \ref{lem:simpl1} to Lemma \ref{lem:lescontributionsnonnulles} imply that
\begin{equation}
\label{eq:BcongrA}
B_{n+1,k-1} - B_{n+1,n-k} = A_{n+1,k-1} - A_{n+1,n-k} + x(A_{n,k-1}-A_{n,n-k})
\end{equation}
for all $k \in [n]$. Now, at this step we know that Formula \hyperref[eq:A]{$(13_{n,k})$} and Formula \hyperref[eq:A]{$(13_{n+1,k})$} are true, which implies that
\begin{align*}
A_{n,k-1} - A_{n,n-k} &= -\sum_{i=1}^{k-1} K_{n-1,i} + \sum_{i=1}^{n-k} K_{n-1,i},\\
A_{n+1,k-1} - A_{n+1,n-k} &= -\sum_{i=1}^{k-1} K_{n,i} + \sum_{i=1}^{n-k} K_{n,i},
\end{align*}
so Formula (\ref{eq:BcongrA}) implies Formula \hyperref[eq:B]{$(14_{n+1,k})$} for all $k \in [n]$ in view of Formula (\ref{eq:inductionK}) and Formula (\ref{eq:symetrieK}). This ends the proof of Theorem \ref{theo:mastertheo}.

\section{Open problems}
\label{sec:open}

Conjecture \ref{conjecture} proved by Theorem \ref{theo:mastertheo} is a particular case of the following conjecture, as we explain afterwards.

\begin{conjecture}[Lando,2016]
\label{conj2}
The generating function $\sum_{t \geq 0} D_n(x) t^n$ has the continued fraction expansion
\begin{equation*}
\label{eq:continuedfraction}
\dfrac{1}{1-b_0(x)t - \dfrac{\lambda_1(x) t^2}{1 - b_1(x) t - \dfrac{\lambda_2(x)t^2}{\ddots}}}
\end{equation*}
where $b_k(x) = x - k(k+1)$ and $\lambda_k(x) = -k^2 x + \binom{k}{2} \binom{k+1}{2}$.
\end{conjecture}

Following Flajolet's theory of continued fractions~\cite{Flajolet}, recall that a Motzkin path of length $n \geq 0$ is a tuple $(p_0,\hdots,p_n) \in ([0,n] \times [0,n])^n$ such that $p_0~=~(0,0)$, $p_n = (n,0)$ and $\overrightarrow{p_{i-1} p_{i}}$ equals either $(1,1)$ (we then say it is an \textit{up step}), or $(1,0)$ (an \textit{horizontal step}), or $(1,-1)$ (a \textit{down step}), for all $i \in [n]$. Conjecture (\ref{conj2}) is equivalent to
\begin{equation*}
D_n(x) = \sum_{\gamma \in M_n} \omega_{b_{\bullet}(x),\lambda_{\bullet}(x)}(\gamma)
\end{equation*}
for all $n \geq 0$, where $\omega_{b_{\bullet}(x),\lambda_{\bullet}(x)}(\gamma)$ is the product of the weigths of the steps of $\gamma~\in M_n$, where an up step is weighted by $1$, an horizontal step from $(x,y)$ to $(x+1,y)$ by $b_y(x)$, and a down step from $(x,y)$ to $(x+1,y-1)$ by $\lambda_y(x)$.

Now, for all $n \geq 2$, if $M'_n$ is the subset of the paths $\gamma = (p_0,\hdots,p_n) \in M_n$ whose only points $p_i = (x_i,y_i)$ such that $y_i = 0$ are $p_0$ and $p_n$, then it is clear that
\begin{align*}
\sum_{\gamma \in M_n} \omega_{b_{\bullet}(x),\lambda_{\bullet}(x)}(\gamma) &\equiv \sum_{\gamma \in M'_n} \omega_{b_{\bullet}(x),\lambda_{\bullet}(x)}(\gamma) \mod x^2,\\
& = -x \sum_{\gamma \in M_{n-2}} \omega_{b'_{\bullet}(x),\lambda'_{\bullet}(x)}(\gamma),\\
& \equiv -x \sum_{\gamma \in M_{n-2}} \omega_{\beta_{\bullet},\Lambda_{\bullet}}(\gamma) \mod x^2
\end{align*}
where
\begin{align*}
b'_k(x) &= b_{k+1}(x) \equiv \beta_k = -(k+1)(k+2) \mod x,\\
\lambda'_k(x) &= \lambda_{k+1}(x) \equiv \Lambda_k  = \binom{k+1}{2} \binom{k+2}{2} \mod x.
\end{align*}
Conjecture \ref{conjecture} is then a particular case of Conjecture \ref{conj2} in that
$$\sum_{n\geq 0}(-1)^n h_{n+1} t^n = \dfrac{1}{1-\beta_0 t-\dfrac{\Lambda_1 t^2}{1-\beta_1 t - \dfrac{\Lambda_2 t^2}{\ddots}}},$$
which we can obtain by applying Lemma \ref{lem:transf} hereafter on the following formula~(see~\cite{HZ,Feigin2})~:
$$\sum_{n\geq 0}(-1)^n h_{n} t^n = \dfrac{1}{1-\dfrac{-\binom{2}{2} t}{1- \dfrac{-\binom{2}{2} t}{1-\dfrac{-\binom{3}{2} t}{1- \dfrac{-\binom{3}{2}t}{1-\dfrac{-\binom{4}{2}t}{\ddots}}}}}}.$$
\begin{lemma}[Dumont and Zeng~\cite{DZ}]
\label{lem:transf}
Let $(c_n)_{n\geq 0}$ be a sequence of complex numbers, then
$$\dfrac{c_0}{1-\dfrac{c_1 t}{1- \dfrac{c_2 t}{\ddots}}} = c_0 + \dfrac{c_0c_1t}{1-(c_1+c_2)t-\dfrac{c_2c_3 t^2}{1- (c_3+c_4)t-\dfrac{c_4 c_5 t^2}{\ddots}}}.$$
\end{lemma}

Another ambitious problem would be to extend the combinatorial interpretations provided by Theorem \ref{theo:mastertheo} to any chord diagram.

\section*{Acknowledgements}
I thank Evgeny Feigin, Sergei Lando and Jiang Zeng for their valuable information and suggestions.


\begin{thebibliography}{9}

\bibitem{BD}
D. Barsky and D. Dumont. “Congruences pour les nombres de Genocchi de 2e espèce (French)”. In: \textit{Study Group on Ultrametric
Analysis} 34 (7th–8th years: 1979–1981), pp. 112–129.

\bibitem{Bigeni} 
A. Bigeni. \textit{Combinatorial interpretations of the Kreweras triangle in terms of subset tuples}. 2017. eprint: \href{https://arxiv.org/abs/1712.01929}{arXiv:1712.01929}.

\bibitem{CDM}
S. Chmutov, S. Duzhin and J. Mostovoy. \textit{Introduction to Vassiliev knot invariants}. Cambridge University Press, 2012.

\bibitem{CV}
S. Chmutov and A. Varchenko. “Remarks on the Vassiliev knot invariants coming from $\text{sl}_2$”. 
Topology 36 (1997), no. 1, 153--178. 

\bibitem{Dumont}
D. Dumont. “Interprétations combinatoires des nombres de Genocchi”. In: \textit{Duke Math}. J. 41 (1974), pp. 305–318.

\bibitem{DV}
D. Dumont and G. Viennot. “A combinatorial interpretation of the Seidel generation of Genocchi numbers”. In: \textit{Ann. Discrete
Math}. 6 (1980), pp. 77–87.

\bibitem{DZ}
D. Dumont and J. Zeng. “Further results on the Euler and Genocchi numbers”. In: \textit{Aequationes Math}. 47.1 (1994), pp. 31–42.

\bibitem{Feigin}
E. Feigin. “Degenerate flag varieties and the median Genocchi numbers”. In: \textit{Math. Res. Lett.} 18.6 (2011), pp. 1163–1178.

\bibitem{Feigin2}
E. Feigin. “The median Genocchi numbers, $q$-analogues and continued fractions”. In: \textit{European J. Combin.} 33.8 (2012), pp. 1913–1918.

\bibitem{Flajolet}
P. Flajolet. “Combinatorial aspects of continued fractions”. In: \textit{Discrete Math}. 32.2 (1980), pp. 125–161.

\bibitem{HZ}
G.-N. Han and J. Zeng. “On a q-sequence that generalizes the median Genocchi numbers”. In: \textit{Ann. Sci. Math. Québec} 23.1 (1999),
pp. 63–72.

\bibitem{Kreweras}
G. Kreweras. “Sur les permutations comptées par les nombres de Genocchi de première et deuxième espèce (French)”. In: \textit{European
J. Combin.} 18.1 (1997), pp. 49–58.

\bibitem{KB}
G. Kreweras and J. Barraud. “Anagrammes alternés”. In: \textit{European J. Combin.} 18.8 (1997), pp. 887–891.

\bibitem{LZ}
S. K. Lando and A. K. Zvonkin. \textit{Graphs on surfaces and their applications}. Springer, 2004. Chap. 6.

\bibitem{G}
OEIS Foundation Inc., The On-Line Encyclopedia of Integer Sequences, \url{http://oeis.org/A110501} (2011).

\bibitem{H}
OEIS Foundation Inc., The On-Line Encyclopedia of Integer Sequences, \url{http://oeis.org/A005439} (2011).

\bibitem{h}
OEIS Foundation Inc., The On-Line Encyclopedia of Integer Sequences, \url{http://oeis.org/A000366} (2011).

\end{thebibliography}
\end{document}